\documentclass[12pt]{amsart}
\usepackage{amsfonts}
\usepackage{amsmath, amscd, amssymb, yhmath}
\usepackage{url}
\usepackage{tikz}
\usepackage{caption}
\usepackage{subcaption}
\usepackage{graphicx}
\usepackage{adjustbox}
\usepackage{multicol}
\usepackage{svg}
\usepackage[colorlinks=true,linkcolor=black,anchorcolor=black,citecolor=black,filecolor=black,menucolor=black,runcolor=black,urlcolor=black]{hyperref}
\newtheorem{theorem}{Theorem}[section]
\newtheorem{lemma}{Lemma}[section]

\begin{document}

\title{The $H(n)$-move is an unknotting operation for virtual and welded links}

\author{Danish Ali}
\address{Department of Mathematics, Dalian University of Technology, China}
\email{danishali@mail.dlut.edu.cn}

\author{Zhiqing Yang}
\address{Department of Mathematics, Dalian University of Technology, China}
\email{yangzhq@dlut.edu.cn}

\author{Abid Hussain}
\address{School of Computing and  Artificial Intelligence Southwest Jiaotong University, China}
\email{abidhussain@my.swjtu.edu.cn}

\author{Mohd Ibrahim Sheikh}
\address{Department of Mathematics, Graduate School of Natural Sciences, Pusan National University, Busan 46241, Republic of Korea}
\email{ibrahimsheikh@pusan.ac.kr}

\date{\today}

\begin{abstract}
\large{ An unknotting operation is a local move such that any knot diagram can be transformed into a diagram of the trivial knot by a finite sequence of these operations plus some Reidemeister moves. It is known that for all $n \geq 2$ the $H(n)$-move is an unknotting operation for classical knots and links. In this paper, we extend the classical unknotting operation $H(n)$-move to virtual knots and links. Virtualization and forbidden move are well-known unknotting operations for virtual knots and links. We also show that virtualization and forbidden move can be realized by a finite sequence of generalized Reidemeister moves and $H(n)$-moves. 
}
\end{abstract}

\maketitle

\large{
\section{Introduction}
Virtual knots were first introduced by Kauffman \cite{Kauffman1}. A virtual knot is a generalization of knots in thickened surfaces. A virtual link diagram has two types of crossings: classical crossings and virtual crossings. Classical crossings (positive and negative) and a virtual crossing are shown in Fig.\ref{cavc}. A classical crossing has an over arc and an under arc. While a virtual crossing has neither over arc nor under arc. A virtual crossing consists of two arcs, with a small circle placed around the double point. Two classical knot diagrams $K$ and $K'$ are equivalent if and only if $K$ can be transformed into $K'$ using classical Reidemeister moves ($R1-R3$) as shown in Fig.\ref{crm}. Similarly, the two virtual knot diagrams $V$ and $V'$ are equivalent if and only if $V$ can be transformed into $V'$ using classical Reidemeister moves and virtual Reidemeister moves ($VR1-VR4$) as shown in Fig.\ref{vrm}. Classical Reidemeister moves and virtual Reidemeister moves are known as generalized Reidemeister moves.

\begin{figure}
     \centering
      \begin{subfigure}{0.3\textwidth}
      \centering
         \includegraphics[width=0.6\textwidth]{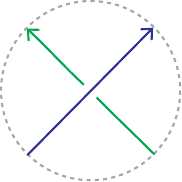}
         \caption{Positive crossing}
         \end{subfigure}
     \hfill
      \begin{subfigure}{0.3\textwidth}
       \centering
 \includegraphics[width=0.6\textwidth]{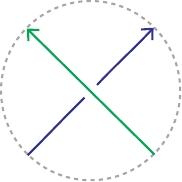}
 \caption{Negative crossing}
     \end{subfigure}
      \hfill
      \begin{subfigure}{0.3\textwidth}
       \centering
 \includegraphics[width=0.6\textwidth]{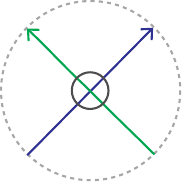}
 \caption{Virtual crossing}
     \end{subfigure}
    \caption{Classical crossings and a virtual crossing}
        \label{cavc}
\end{figure}

\begin{figure}
\includegraphics[page=1, width=0.7\textwidth]{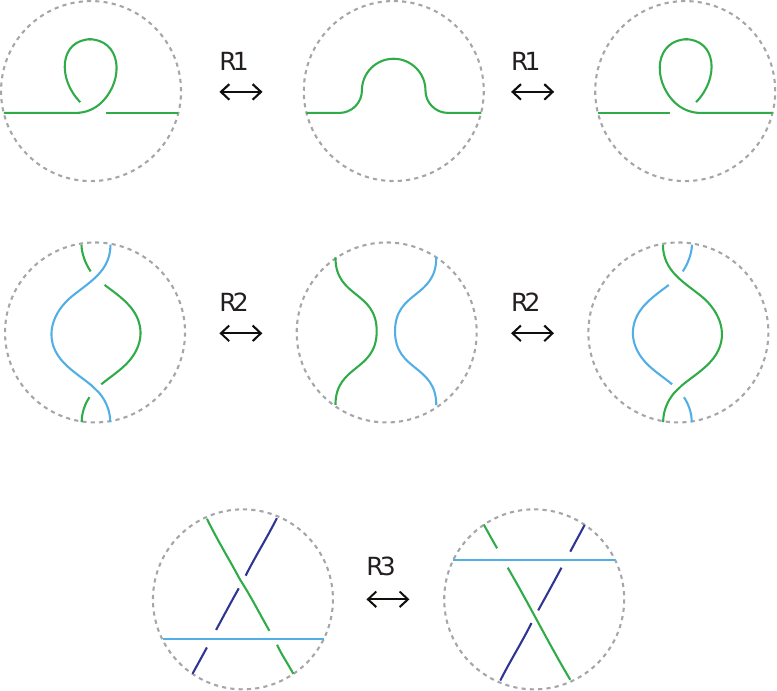}
\caption{Classical Reidemeister moves}
\label{crm}
\end{figure}

\begin{figure}
\includegraphics[width=\textwidth]{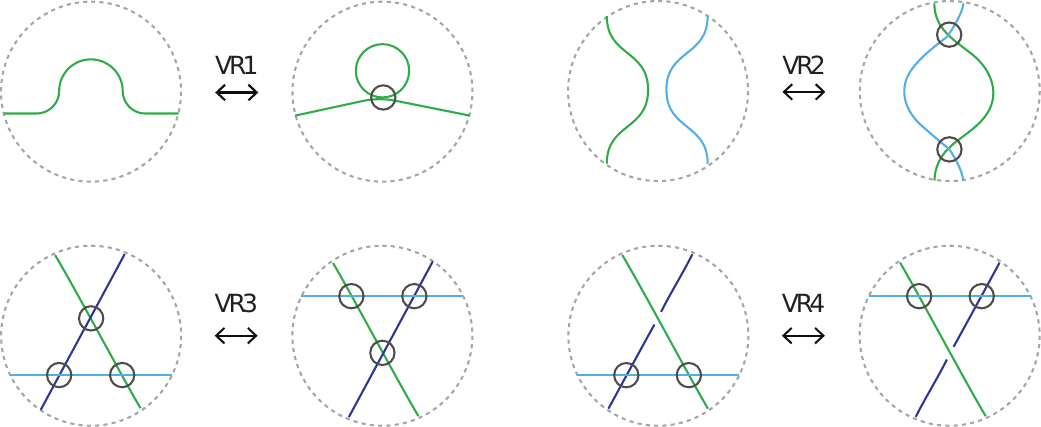}
\caption{Virtual Reidemeister moves and a mixed Reidemeister move}
\label{vrm}
\end{figure}

A local move is a transformation of link diagrams $D$ to $D'$ that are identical except inside one small disk of each. An unknotting operation is a local change on knot (link) diagrams. Any knot (link) diagram can be transformed into a trivial unknot (unlink) diagram by a series of unknotting operations plus some Reidemeister moves. One main focus of knot theory is how to unknot a given knot, and how to distinguish different knots using knot invariants. Many knot invariants are related to unknotting operations, therefore, unknotting operations are very important in knot theory. Many local moves are known as unknotting operations for classical knots and links. Some well-known unknotting operations for classical knots and links are the crossing change, $\Delta$-move \cite{hm}, $\sharp$-move \cite{hm2}, $\Gamma$-move \cite{ts} and $n$-gon move \cite{yn, ha}. 

In classical knot theory, the crossing change is a fundamental technique to unknot any given knot. Any classical knot diagram can be converted into a descending diagram by the following procedure. Choose a base point which is an arbitrary point on a knot diagram, distinct from all crossing points. Now we travel along the knot from the base point according to the orientation of the knot. When we meet a crossing, if we first pass the crossing from the lower arc, we apply a crossing change here; if we first pass the crossing from the upper arc, we skip it. Finally, we get back to the initial point. After making all of these crossing changes, the resulting diagram is called a descending diagram. A descending classical knot diagram represents the unknot. However, crossing change is not an unknotting operation for virtual knots because some descending virtual knot diagrams are not the unknot. On the other hand, the $\Delta$-move, $\sharp$-move, $\Gamma$-move, and $n$-gon move can be realized by a finite sequence of crossing changes, therefore, none of them is an unknotting operation for virtual knots. Crossing virtualization is a local move on a classical crossing that change a classical crossing into a virtual crossing as shown in Fig.\ref{virt}. Crossing virtualization is an unknotting move for virtual knots and links \cite{yo, ham}.

Hoste, Nakanishi and Taniyama in \cite{jh} introduced the $H(n)$-move for $n \geq 2$, as shown in Fig.\ref{hnm} . They proved that the $H(n)$-move is an unknotting operation for classical knots and links. Therefore, any two classical knot diagrams can be transformed into each other by a finite sequence of $H(n)$-moves plus some Reidemeister moves. An $H(n)$-move preserves the number of link components. They defined an $H(n)$-move for non-oriented link diagrams, and they defined an $SH(n)$-move as a special type of $H(n)$-move preserving the orientation of arcs outside the trivial tangle. From now on, unless otherwise stated, we use the name $H(n)$-move for both oriented and non-oriented diagrams of virtual knots and links in this paper. We shall prove that the $H(n)$-move is an unknotting operation for oriented diagrams of virtual knots and links. All the results of oriented diagrams are valid for non-oriented diagrams of virtual knots and links. We shall also prove that the $H(n)$-move is an unknotting operation for welded knots and links.

\begin{figure}
\includegraphics[width=0.46\textwidth]{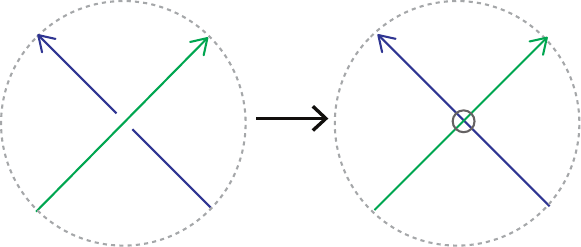}
\caption{Crossing virtualization}\label{virt}
\end{figure} 

\begin{figure}
\centering
\includegraphics[width=0.7\textwidth]{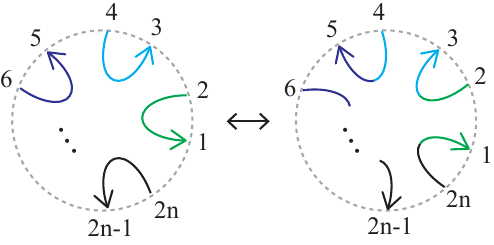}
 \caption{An $H(n)-move$}
 \label{hnm}
\end{figure}

\section{$H(n)$-move is unknotting operation for virtual links}
Let $V$ be a virtual knot diagram. If $V$ has no classical crossings, then using the virtual Reidemeister moves $(VR1-VR3)$ we can transform $V$ into the diagram of the unknot with no crossing. If $V$ has at least one classical crossing, we shall show that any classical crossing can be transformed into a virtual crossing using generalized Reidemeister moves together with $H(n)$-moves.

\begin{theorem}\label{th1}
Given any integer $n (\geq 2)$, any virtual knot diagram  can be transformed into a trivial knot diagram by a finite sequence of $H(n)$-moves plus some generalized Reidemeister moves.
\end{theorem}

\begin{lemma}\label{lem1}
A crossing virtualization can be realized by a finite sequence of $H(2)$-moves plus some generalized Reidemeister moves. The $H(2)$-move is an unknotting operation for virtual knots.
\end{lemma}

\begin{proof}
Suppose $V$ is a diagram of a virtual knot with some classical crossings. Orient the diagram and perform an $H(2)$-move at any classical crossing as shown in Fig.\ref{ctv}. We can replace a classical crossing by a virtual crossing using a finite sequence of $H(2)$-moves plus some generalized Reidemeister moves. All classical crossings of $V$ can be transformed into virtual crossings by $H(2)$-moves. As a result, $V$ becomes a trivial knot diagram.
\end{proof}

\begin{figure}
\centering
\includegraphics[width=0.6\textwidth]{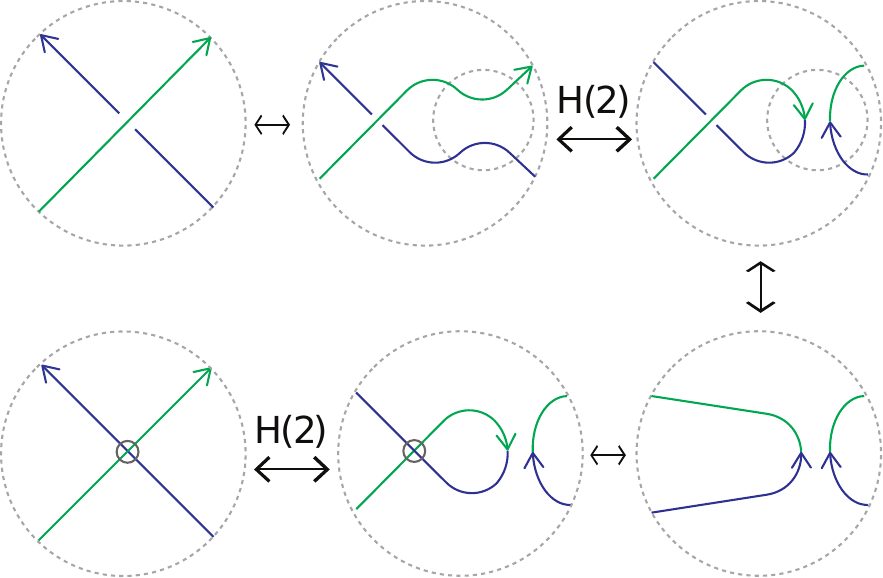}
 \caption{A crossing virtualization can be realized by $H(2)$-moves}
 \label{ctv}
\end{figure} 

\begin{figure}
\centering
\includegraphics[width=0.8\textwidth]{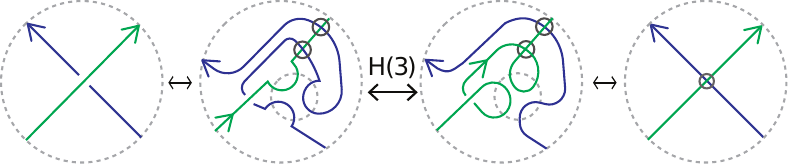}
 \caption{A crossing virtualization can be realized by $H(3)$-moves}
 \label{ctv3}
\end{figure} 

\begin{lemma}\label{lem2}
A crossing virtualization can be realized by a finite sequence of $H(3)$-moves plus some generalized Reidemeister moves.
\end{lemma}

\begin{proof}
Let $V$ be a virtual knot diagram and suppose that $V$ has some classical crossings. We can replace a classical crossing by a virtual crossing using a finite sequence of $H(3)$-moves plus some generalized Reidemeister moves, as shown n Fig.\ref{ctv3}. Therefore, a crossing virtualization can be realized by a finite sequence of $H(3)$-moves plus some generalized Reidemeister moves. Finally, if $V$ has only virtual crossings, then using the virtual Reidemeister moves $(VR1-VR3)$ we can transform $V$ into the diagram of the unknot with no crossing.
\end{proof}

\begin{lemma}\label{lem3}
An $H(n)$-move can be realized by an $H(n+2)$-move.
\end{lemma}

\begin{proof}
In Fig.\ref{hn+2} we illustrate that an $H(n)$-move can be realized by an $H(n+2)$-move.
\end{proof}

\begin{figure}
\centering
\includegraphics[width=0.9\textwidth]{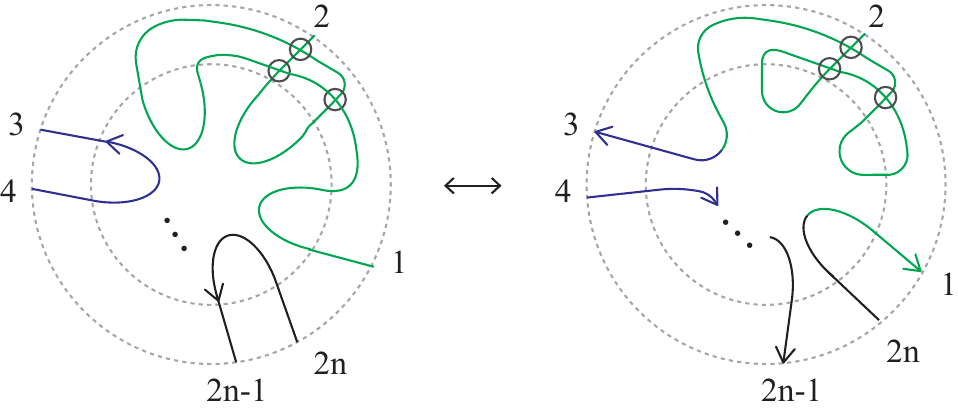}
 \caption{An $H(n)$-move can be realized by an $H(n+2)$-move}
 \label{hn+2}
\end{figure} 

\noindent \textit{Proof of Theorem \ref{th1}.}

\noindent Using Lemmas \ref{lem1}, \ref{lem2} and \ref{lem3}, we can conclude that a crossing virtualization can be realized by a finite sequence of $H(n)$-moves plus generalized Reidemeister moves. Therefore, any classical crossing can be transformed into a virtual crossing using generalized Reidemeister moves together with $H(n)$-moves. All classical crossings of  a virtual knot diagram can be transformed into virtual crossings by $H(n)$-moves. As a result, the virtual knot diagram become a trivial knot diagram. Therefore, the $H(n)$-move is unknotting operation for virtual knots. Similarly, a virtual link diagram, $V=v_1 \cup v_2 \cup \cdots v_n $ with $n$ components, can be unlinked by generalized Reidemeister moves together with $H(n)$-moves. $\square$

In virtual knot theory, unlike generalized Reidemeister moves there are two other moves that are not allowed. The two moves shown in Fig.\ref{fm}. are called forbidden moves, they are known as the $F_O$ (over crossing) move and $F_U$ (under crossing) move. Sam Nelson \cite{sn} and Taizo kanenobu \cite{tk} showed that the two forbidden moves together can unknot all virtual knot diagrams. Taizo kanenobu also showed that a $\Delta$-unknotting operation is realized by a finite sequence of generalized Reidemeister moves and forbidden moves.

If we allow the $F_O$ move then virtual knot theory become welded knot theory. If both forbidden moves are allowed then it become the trivial theory, since any knot in this theory is equivalent to the trivial knot. Generalized Reidemeister moves along with $F_O$ move are called welded Reidemeister moves. Two welded knot diagrams $W$ and $W'$ are equivalent if and only if $W$ can be transformed into $W'$ using welded Reidemeister moves. Welded knot theory is closely related to virtual knot theory therefore it is important to check the effect of $H(n)$-move on welded links.

\begin{figure}
     \centering
      \begin{subfigure}{0.4\textwidth}
      \centering
         \includegraphics[width=\textwidth]{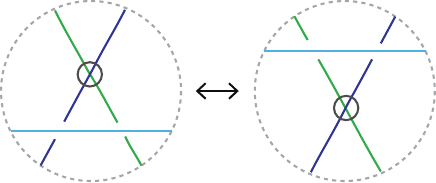}
         \caption{$F_O$ move}
         \end{subfigure}
     \hspace{0.5cm}
      \begin{subfigure}{0.4\textwidth}
       \centering
 \includegraphics[width=\textwidth]{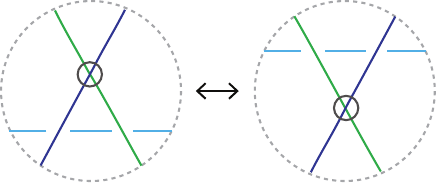}
 \caption{$F_U$ move}
     \end{subfigure}
\caption{The forbidden moves}
\label{fm}
\end{figure}

Since the $H(n)$-move is an unknotting operation for virtual knots, it is also an unknotting operation for welded knots. The $F_U$  move is an unknotting operation for welded knots. We will show that the $F_U$  move can be realized by a finite sequence of $H(n)$-moves plus welded Reidemeister moves.

\begin{theorem}
An $F_U$ move can be realized by a finite sequence of $H(n)$-moves plus welded Reidemeister moves.
\end{theorem}

\begin{proof}
An $F_U$ move can be realized by a finite sequence of $H(n)$-moves plus welded Reidemeister moves as shown in Fig.\ref{fur}.
\end{proof}

\begin{figure}
\centering
\includegraphics[page=1, width=0.6\textwidth]{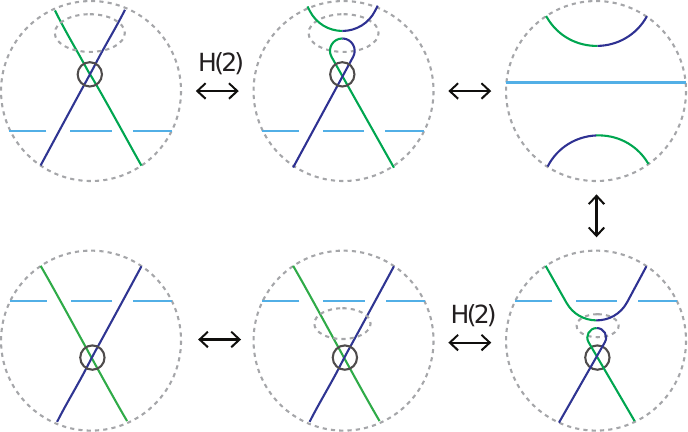}
\caption{An $F_U$ move can be realized by a finite sequence of $H(n)$-moves}
\label{fur}
\end{figure}

\begin{figure}
\centering
\includegraphics[page=1, width=0.4\textwidth]{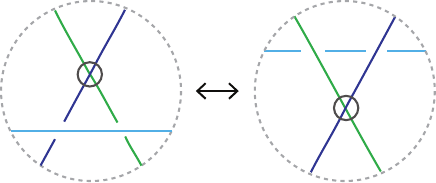}
\caption{A $CF$-move}
\label{cfm}
\end{figure}

\begin{figure}
\centering
\includegraphics[page=1, width=0.6\textwidth]{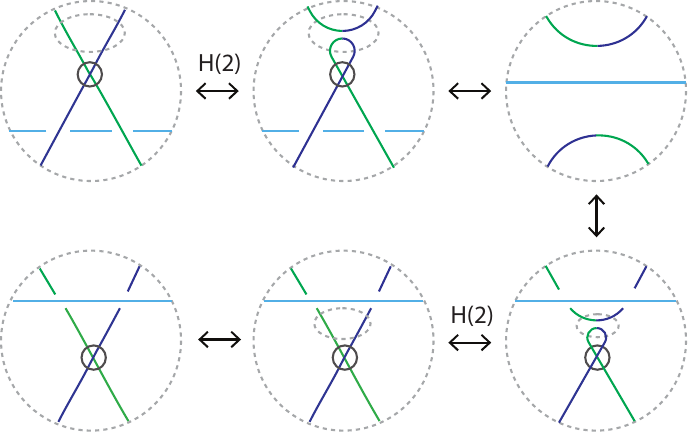}
\caption{An $CF$-move can be realized by a finite sequence of $H(n)$-moves}
\label{cfr}
\end{figure}

The $CF$-move introduced by Oikawa in \cite{to} is a local move on virtual link diagrams as shown in Fig.\ref{cfm}. It is defined as a combination of crossing changes and a forbidden move. Two virtual links are $CF$-equivalent if their diagrams are related by a finite sequence of generalized Reidemeister moves and $CF$-moves. Oikawa proved that any virtual knot is $CF$-equivalent to the trivial knot, i.e. the $CF$-move is an unknotting operation for virtual knots.

\begin{theorem}
A $CF$-move can be realized by a finite sequence of $H(n)$-moves plus some generalized Reidemeister moves.
\end{theorem}

\begin{proof}
A $CF$-move move can be realized by a finite sequence of $H(n)$-moves plus some generalized Reidemeister moves as shown in Fig.\ref{cfr}.
\end{proof}

\bibliographystyle{amsplain}
}
\end{document}